%% file: main.tex
\documentclass{amsart}

\usepackage{amsmath}
\usepackage{amssymb}
\usepackage{amsthm}
\usepackage{hyperref} 
\usepackage{svg}

\pdfoutput=1

\newtheorem{theorem}{Theorem}

\newtheorem{remark}[theorem]{Remark}
\newtheorem{corollary}[theorem]{Corollary}
\newtheorem{lemma}[theorem]{Lemma}
\newtheorem{definition}[theorem]{Definition}
\newtheorem{example}[theorem]{Example}

\begin{document}

\title{Planarity of Mycielski-like graph expansions}

\author{Devansh Vimal}
\address{Biodesign Center for Biocomputing, Security and Society, Arizona State University, Tempe, Arizona, 85281}
\email{dvp0@asu.edu}
\urladdr{agentelement.net}

\begin{abstract}
For a graph $G$, we define its great shadow $S(G)$ as a construction that
duplicates each vertex $v$ in $G$ and sets this duplicated vertex
adjacent to $v$ and all neighbors of $v$.
Great graph shadows arise naturally in the routing of diode-and-switch circuits for computer keyboards,
and are closely related to the Mycielski operation.
These diode-and-switch circuits can be routed on a single-sided printed-circuit board
if and only if the corresponding great shadow is planar.
In this paper, we characterize all graphs with planar great shadows.
Such graphs are always bipartite cactus graphs.
\end{abstract}

\maketitle

\section{Introduction}\label{sec:intro}
The \emph{great shadow}, $S(G)$,
of a graph $G$ duplicates every vertex $v$ in $G$
and sets this duplicated vertex incident to $v$ and all neighbors of $v$.
The great shadow is similar to Mycielski's construction \cite{mycielski1955},
a construction well-studied for the property of increasing the
chromatic number of a graph while also keeping the graph triangle-free
(if $G$ is triange-free).
Mycielski's construction duplicates each vertex $v \in G$
and introduces one special vertex $c$.
Then, it sets each duplicated vertex incident to
all neighbors of $v$ and incident to $c$.
In $S(G)$, we omit the vertex $c$,
and set the duplicated vertex incident to $v$ instead.
Our main contribution of this paper is a characterization of all graphs with planar great shadows.

\begin{theorem}\label{thm:main}
    Let $G$ be a connected graph.
    Then $S(G)$ is planar if and only if $G$ is a bipartite cactus graph.
\end{theorem}

Cactus graphs are graphs where no two cycles share an edge.
In a bipartite cactus, every cycle is of even order.

\subsection{Origin of the problem}\label{sec:origin}
Electrical circuits may be represented by undirected graphs
with components on their vertices or edges.
Circuits are commonly routed on printed-circuit boards with multiple layers,
and such boards are cheapest to manufacture when they have as few layers as possible.
If the graph of a circuit is planar,
then the circuit can be routed on a single-sided printed-circuit board,
minimizing its manufacturing cost.
Our problem arises in the context of designing inexpensive mechanical computer keyboards.

Computer keyboards are simple circuits,
and in our model,
they consist of
a single microcontroller with $n$ pins connected in some manner to
$m$ diodes and $m$ electrical button switches.
An edge in the graph of our circuit is either a conductive wire,
or a diode and switch in series.
Multiple edges are allowed between two vertices,
and a vertex may be attached to a microcontroller pin.
Each pin of the microcontroller has three states: HIGH, LOW, and INPUT.
An edge $u v$ is \emph{conductive} if either it is a conductive wire, or if
there is a diode with its anode connected to $u$, cathode connected to $v$,
and the switch on $u v$ is depressed.
Two nonadjacent vertices $x$ and  $y$ are on a conductive path
if there is a path $x u_1, \dots, u_\ell y$ of conductive edges.
When a pin of a microcontroller on vertex $v$ is set HIGH,
all microcontroller pins on a vertex with a conductive path from $v$
and set to the INPUT state recieve the HIGH signal.
Otherwise, all microcontroller pins set to the INPUT state recieve a LOW signal.

The microcontroller must first simultaneously set all of its pins
to the HIGH, LOW, or INPUT states,
then may read the values of all pins set to the INPUT state
before performing any computation on the resulting input.
With this set-up,
the microcontroller must respond to the following query:
Is switch $1 \leq i \leq m$ pressed?
In some circuits,
it may not be possible to uniquely distinguish two simultaneously pressed switches,
so here we consider only circuits where the query can be answered for all switches.
Consequently,
we are limited to at most two switches per pair of vertices
(as we can have at most two edges per pair of vertices,
each with switch and a diode in either direction),
and a maximum of $n (n - 1)$ switches
(as a maximally dense circuit takes the form of a clique).\footnote{
    The circuit routing scheme with $n (n - 1)$ switches is called \emph{Charlieplexing},
    first described by Maxim electronics for their MAX6951 LED display driver
    \cite{lancaster2001musing,maxim2003reduced}.}
The state of every switch
is obtained by repeating the query for all $m$ switches,
a procedure called a \emph{scan}.
In order to ensure that the delay between the press of a switch
and its detection between the microcontroller is imperceptible to a typist,
scans may be executed several hundred times per second.

However, repeatedly scanning the circuit is power-inefficent.
Electrical engineers solve this by introducing an \emph{interrupt} pin:
a special pin, that when set HIGH,
triggers a single scan.
With an interrupt pin,
the microcontroller need not continuously scan,
and instead may idle when no switch is being pressed.
Introducing this interrupt pin requires a modification of the original circuit.
We will observe that if the graph of the original circuit is $G$,
then the graph of the circuit with the interrupt pin is $S(G)$.

Suppose each pair of connected vertices in $G$ has both directions of switches between them.
Each edge that directs a HIGH signal (when that edge is conductive) to any vertex $v$
must also direct that signal to the interrupt pin.
However,
we cannot connect the interrupt pin by conductive edges directly to $v$,
as $v$ may go HIGH itself during a scan.
Therefore,
the interrupt pin must be connected by conductive edges
to an intermediate vertex $v'$
such that all edges that are directed toward $v$ in $G$
are directed towards $v'$ in the interrupted circuit instead,
and an edge with a diode but no switch is inserted between $v$ and $v'$.
Therefore, for each vertex $v$,
one creates an intermediate vertex $v'$
that is connected to each neighbor of $v$ by a switch and diode,
and is connected to $v$ by only a diode.
This construction is precisely $S(G)$.
Additionally,
each vertex $v'$ may be connected by conductive edges
to a single vertex attached to the interrupt pin.
We omit these edges from $S(G)$,
as one can just place multiple interrupt pins on each $v'$.
But we note that the connection to the Mycielskian of $G$
is even stronger with these additional edges:
the additional edges all connect the vertices $v'$ to a central vertex,
so this construction differs from the Mycielskian only because it has
edges between each $v$ and $v'$.

\subsection{Related work}\label{sec:relwork}
An ordinary shadow of $G$ --- which we call $s(G)$ ---
is a construction that duplicates every vertex $v$ in $G$
and sets this duplicated vertex incident to all neigbors of $v$.
The shadow is a subgraph of the great shadow, as
edges between $v$ and $v'$ are absent in the shadow but present in the great shadow.
The Mycielskian of $G$ without its central vertex produces $s(G)$.
The planarity of $s(G)$ is studied by Garza and Shinkel \cite{garza1999shadow}.
They give a complete characterization of graphs with planar shadows.

\begin{theorem}[Garza and Shinkel \cite{garza1999shadow}]\label{thm:smallshadow}
    Let $G$ be a nontrivial connected graph.
    Then $s(G)$ is planar if and only if
    every block of $G$ is $K_2$, $K_3$, $K_4^-$, $K_4$, or an even cycle,
    and $G$ has the following properties:
    (1) every cut vertex of $G$ has degree at most 2 in every block containing it,
    and
    (2) if $K_3$ is a block of $G$ then not all three vertices of the block are cut vertices of $G$.
\end{theorem}

Cactus graphs have been studied in other contexts around electrical circuits,
particularly circuits of operational amplifiers \cite{nishi1986uniqueness,nishi1986topological}.

\section{Notation}\label{sec:notation}
All graphs considered hereafter are simple and loopless.
If $G$ and $H$ are graphs, we write $H \subseteq G$ to mean that $H$ is a
(not necessarily induced) subgraph of $G$.
We write $T G$ to denote any topological graph obtained from $G$,
constructed by subdividing the edges of $G$.
For any two graphs $G$ and $G'$,
if $T G$ is a subgraph of $G'$,
then we say that $G$ is a minor of $G'$.
We write $K_n$ to denote the complete graph on $n$ vertices,
$P_n$ to denote a path graph with $n$ vertices,
and $K_{m,n}$ to denote a complete bipartite graph
with one partition having $m$ vertices
and the other partition having $n$ vertices.
Of special interest to us are the graphs
$K_{3, 3}$ and $K_4^-$
($K_4$ minus one edge, also commonly known as the diamond graph).
A cactus graph is a connected graph where each $2$-connected component is a cycle.
Equivalently,
a cactus graph is a graph where no two cycles share an edge \cite{el1988complexity}.
If a graph $G$ has a vertex $v$ such that $G - v$ is disconnected, then
$v$ is called a cut vertex.

\section{Preliminaries}\label{sec:prelim}

First, we formally define the great shadow.
\begin{definition}[Great shadow graph, $S(G)$]\label{def:shadow}
    Let $G$ be a graph on $n$ vertices. Define $S(G)$ as a graph on $2n$ vertices.
    For each vertex $v$,
    introduce an associate vertex $v'$,
    called the shadow vertex of $v$,
    and set $v'$ adjacent to $v$ and all neighbors of $v$ in $G$.
\end{definition}

The great shadow operation has a few properties that we use throughout the paper.
We note that two shadow vertices are never adjacent.
Also, for any two graphs $G$ and $H$, if $H \subseteq G$ then $S(H) \subseteq S(G)$.
We prove this below.

\begin{lemma}\label{lem:subgraph}
    If $H \subseteq G$ then $S(H) \subseteq S(G)$.
\end{lemma}
\begin{proof}
    We have $H \subseteq G \subseteq S(G)$,
    so it suffices to determine
    $S(H) - E(H) \subseteq S(G) - E(G)$.
    Let $v \in H \subseteq G$ and $v' \in S(H) \subseteq S(G)$.
    As $H \subseteq G$,
    we have
    $N_H (v) \subseteq N_G (v)$,
    and therefore $N_{S(H)}(v') \subseteq N_{S(G)}(v')$.
    All edges in $S(H) - E(H)$ take the form
    $v' n$ for some $n \in N_{S(H)} (v')$,
    and therefore there is a corresponding edge $v' n \in S(G) - E(G)$.
\end{proof}

We invoke Kuratowski's theorem \cite{kuratowski1930probleme}
to prove the non-planarity of a given graph.
His theorem gives a characterization of planar graphs
by a set of forbidden subgraphs.
If we find any such subgraph in a graph that we are studying,
then that subgraph is a \emph{witness} of non-planarity.
\begin{theorem}[Kuratowski \cite{kuratowski1930probleme}]\label{thm:kuratowski}
    A graph is planar if and only if it contains neither $T K_5$ nor
    $T K_{3, 3}$ as a subgraph.
\end{theorem}

Bipartite graphs and cactus graphs can both be characterized by
families of forbidden subgraphs.
The following result is folklore,
and a proof can be found in any standard textbook (for example, \cite{diestel2010graph}).
\begin{lemma}\label{lem:bipartiteforbidden}
    A graph is bipartite if and only if it contains no odd cycles.
\end{lemma}

Cactus graphs have several equivalent definitions.
A graph is cactus if and only if no two cycles share an edge,
or if and only if every two-connected component is a cycle.
We use the following forbidden graph characterization, from \cite{el1988complexity}.
\begin{lemma}\label{lem:cactusforbidden}
    A graph is cactus if and only if it does not contain $T K_4^-$ as a subgraph.
\end{lemma}

Cactus graphs have a natural tree-like structure.
A cactus graph can be recursively defined in a manner similar to a tree.
It is easy to see that this definition produces all graphs where no two
cycles share an edge (and thus all cacti):
\begin{definition}
    (1) An isolated vertex or a cycle is a cactus graph.
    (2) Let $G$ be a cactus graph.
    Attach anywhere to $G$ a pendant vertex or a pendant cycle
    (a cycle touching $G$ at one vertex), producing $G'$.
    Then $G'$ is also a cactus graph.
\end{definition}

Given some cactus, we may obtain its \emph{cycle tree}
by collapsing all cycles to a single representative vertex.
To transform a cactus into its cycle tree,
we introduce a representative vertex for each cycle.
Then, we delete all degree-two vertices participating in any cycle,
thus deleting all non-cut vertices of the cactus.
We obtain the cycle tree by connecting each surviving vertex formerly
participating a cycle to its representative vertex.
An induction on the cycle tree
is used to prove the sufficiency condition of Theorem~\ref{thm:main}.
\begin{definition}[Cycle tree]\label{def:blocktree}
    Let $\mathbb{C}$ denote the set of cycles of a cactus graph $G$,
    let $X$ denote the set of cut vertices of $G$,
    and let $E$ denote the set of edges not participating in any cycle.
    The cycle tree $T$ of $G$ is a graph on
    $X \cup \mathbb{C}$ formed by the edge set
    $ \{ a C : a \in X \cap C , C \in \mathbb{C} \} \cup \{ a b : a, b \in X, a b \in E\} $
\end{definition}

Drawings of planar graphs enclose zero or more regions of the plane.
Such regions are called the internal faces of the drawing.
The external face is the unique unenclosed region containing the point at infinity.
Any planar graph can be redrawn such that
the edges touching an internal face in one drawing
touch an external face in another.
This fact is used in the aforementioned induction.
\begin{lemma}\label{lem:externalface}
     Let $G$ be a planar graph and
     let $\Gamma$ be some drawing of $G$ in the plane.
     Let $F$ be an internal face of $G$ in $\Gamma$ and
     let $E(F)$ be the set of edges adjacent to $F$ in $\Gamma$.
     Then there is a drawing $\Gamma'$ of $G$ with external face $F'$
     such that all $E(F)$ are adjacent to $F'$.
\end{lemma}
\begin{proof}
    We embed the drawing $\Gamma$ onto a sphere by a stereographic projection:
    Place a sphere $S$ on the plane containing $\Gamma$.
    Let $s$ be the point of intersection of the plane and sphere, and call $s$ the South pole.
    Call the point diametrically opposite to $s$ on the sphere the North pole, $n$.
    Then map every point $q$ on the plane onto the
    unique point on $S$ intersecting the line $n q$.

    We obtain $\Gamma'$ by rotating the sphere such that its new North pole $n'$ is contained within $F$,
    then reversing the stereographic projection to draw $G$ back on the plane.
    Every point in $F \setminus \{n'\}$ is mapped to the external face $F'$,
    and so $F'$ has the same edges as $F$, as desired.
\end{proof}

An edge is \emph{exposed} if it touches an external face.
As a corollary of Lemma~\ref{lem:externalface},
any edge of a planar graph can be exposed.
\begin{corollary}\label{cor:exposededge}
    Let $G$ be a planar graph. Let $e$ be an edge of $G$.
    There is a drawing of $G$ on the plane exposing $e$.
\end{corollary}
\begin{proof}
    Let $F$ be some face of $G$ touching $e$ in some drawing of $G$.
    Apply Lemma~\ref{lem:externalface} to $F$ to make $F$ the external face.
    Now $e$ is exposed.
\end{proof}

\section{A sketch of the proof, and some examples}\label{sec:sketch}
In Section~\ref{sec:bipartitecactusnecc}, we show the necessary conditions of Theorem~\ref{thm:main}.
To prove that $S(G)$ is not planar when $G$ is either not bipartite or not a
cactus graph, we show that $S(G)$ contains an induced $T K_{3, 3}$ if either
    \begin{itemize}
        \item $G$ contains an odd cycle (Lemma~\ref{lem:bipartitenecc}), or,
        \item $G$ contains a $T K_4^-$ (Lemma~\ref{lem:cactusnecc}).
    \end{itemize}

Bipartite graphs are characterized by having odd cycles as forbidden subgraphs (Lemma~\ref{lem:bipartiteforbidden}),
and cactus graphs are characterized by having $K_4^-$ as a forbidden minor (Lemma~\ref{lem:cactusforbidden}).
If we show that $G$ must not contain any of these graphs as subgraphs for $S(G)$ to be planar,
then $G$ must be a bipartite cactus graph.

In Section~\ref{sec:bipartitecactussuff},
we show the sufficiency conditions of Theorem~\ref{thm:main}.
This is done by providing a recursive construction for a planar drawing of $S(G)$ when $G$ is a bipartite cactus graph.
The recursion is done on the cycle tree $T$ of $G$:
We first obtain $T$,
then remove one leaf vertex from $T$ to obtain two smaller bipartite cacti $G_1$ and $G_2$.
By induction, we obtain planar drawings of $S(G_1)$ and $S(G_2)$.
Then we `glue' these drawings together on the plane to obtain a planar drawing of $S(G)$.

\begin{example}\label{ex:sk3}
    The graph $K_3$ is not bipartite, and thus $S(K_3)$ is not planar.
    We find a $K_{3, 3}$ subgraph witnessing the non-planarity by taking each
    edge $u v'$, where $v'$ is a shadow vertex.
\end{example}

\begin{figure}[htbp]
    \centering
    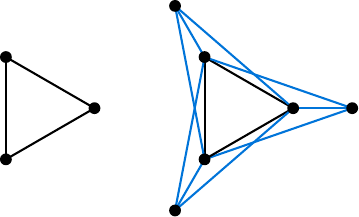
    \caption{The graph $K_3$ is on the left, and its great shadow, $S(K_3)$, is on the right.
        A $K_{3,3}$ subgraph witnessing the non-planarity
        of $S(K_3)$ is highlighted in blue.}
\end{figure}

\begin{example}\label{ex:sk4e}
    The graph $C_4 + v$ is a bipartite cactus graph,
    and indeed $S(C_4 + v)$ is planar.
\end{example}

\begin{figure}[htbp]
    \centering
    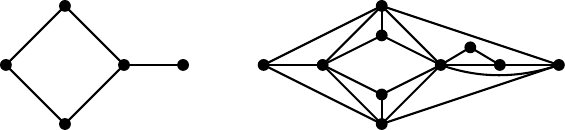
    \caption{The graph $C_4 + v$ is on the left, and a planar drawing of its great shadow is on the right.}
\end{figure}

\section{Bipartite cactus is necessary}\label{sec:bipartitecactusnecc}

We first show that if $G$ contains an odd cycle $C$,
$S(C)$ contains a $T K_{3, 3}$ and therefore $S(G)$ cannot be planar.
Note that a graph is bipartite if and only if it contains no odd cycles.
This fact, along with Lemma~\ref{lem:subgraph}, is sufficient to show that any graph $G$
containing an odd cycle does not have a planar great shadow.
\begin{lemma}\label{lem:bipartitenecc}
    If $G$ contains an odd cycle, then $T K_{3, 3} \subseteq S(G)$.
\end{lemma}
\begin{proof}
    Let $C$ be an odd cycle contained in $G$.
    We will find a $T K_{3, 3}$ subgraph in $S(C)$.
    Then, by Lemma~\ref{lem:subgraph}, $T K_{3, 3} \subseteq S(G)$.
    Cycle $C$ has at least three vertices on a $P_3$, which we name $u, v, w$.
    Their shadows in $S(C)$ are $u', v', w'$.
    Observe that the set $\{u, v, w, u', v', w'\}$ nearly induces a $K_{3, 3}$
    subgraph in $S(G)$, with one partition given by $\{u, v, w\}$ and the other
    by $\{u', v', w'\}$. The only missing edges are $u w'$ and $w' u$.
    We obtain a $T K_{3, 3}$ subgraph by finding
    two disjoint paths to replace these edges.
    First index each vertex in $C$ beginning from $v$ as
    $v = v_0, w = v_1, \dots, v_{2k} = u$.
    Our disjoint paths are
    $ P_1 := w = v_1, v'_2, v_3, \dots, v'_{2k-1}, v'_{2k} = u' $
    and
    $ P_2 := w' = v'_1, v_2, v'_3, \dots, v'_{2k-1}, v_{2k} = u. $
    Both paths alternate between vertices in $C$ and the shadow vertices,
    each never intersecting with the other.
    Path $P_1$ substitutes for $w u'$
    and
    $P_2$ substitutes for $w' u$,
    giving a $T K_{3, 3}$ subgraph.
\end{proof}

\begin{remark}\label{rmk:altbipartitenecc}
    We obtain a short proof of Lemma~\ref{lem:bipartitenecc} if we
    observe that $S(C)$ contains a subdivided square of $C$.
    Square odd cycles of order at least $4$ are not planar \cite{harary1967criterion},
    and $S(K_3)$ is not planar by Example~\ref{ex:sk3},
    so $S(C)$ is not planar when $|C|$ is odd.
\end{remark}

Now we show that if $G$ is not a cactus graph,
it must contain a topological $K_4^-$.
We will show in Lemma~\ref{lem:cactusnecc} that if there is a $T K_4^-$ in $G$,
then there is a $T K_{3, 3}$ in $S(G)$.
Thus if $S(G)$ is planar, then $G$ must be a cactus graph.

\begin{definition}[$\theta(\ell, m, n)$]\label{def:shadowgraph}
    Let $\ell, n \geq 1$ and $m \geq 2$.
    The graph $\theta(\ell, m, n)$ is a graph on $\ell+m+n$ vertices
    obtained by joining the endvertices of a $P_\ell$ and $P_n$
    to the endvertices of a $P_m$.
    Equivalently, $\theta(\ell, m, n)$ is a graph with two fundamental cycles
    of order $\ell + m$ and $m + n$, joined at a path of length $m$.
\end{definition}
\begin{example}
    The diamond graph, $K_4^-$, is $\theta(1, 2, 1)$.
\end{example}

\begin{lemma}\label{lem:cactusnecc}
    If $T K_4^- \subseteq G$ then $T K_{3, 3} \subseteq S(G)$.
\end{lemma}
\begin{proof}
    Every $T K_4^-$ is a $\theta(\ell, m, n)$, so we look for a
    $T K_{3, 3}$ in every $S(\theta(\ell, m, n))$.

    Fix $\ell, m, n$ and let $\theta = \theta(\ell, m, n)$.
    Let $A$ be the fundamental cycle in $\theta$ of order $\ell + m$
    and
    let $B$ be the fundamental cycle in $\theta$ of order $n + m$.
    Let $u$ and $v$ be the endvertices of the shared path between $A$ and $B$.
    Let $a_1, \dots, a_\ell$ be the vertices of $A \setminus B = P_\ell$
    and
    let $b_1, \dots, b_n$ be the vertices of $B \setminus A = P_n$.
    Let $\xi_1, \dots, \xi_{m - 1}$ be the vertices of $A \cap B \setminus \{u, v\}$,
    the vertices on the interior of the shared path $P_m$ between $A$ and $B$.
    We adopt the convention that $a_1, b_1$ and $\xi_1$ are all incident to $u$,
    and the numbering of
    $a_1, \dots, a_\ell$, of $b_1, \dots, b_n$, and of $\xi_1, \dots, \xi_{m - 2}$
    follows each path $P_\ell$, $P_n$ and $P_m$ such that
    $a_\ell, b_n$ and $\xi_{m - 2}$
    are set incident to $v$.

    If any pairwise sum of $\ell, m, n$ is odd, then
    $\theta$ contains an odd cycle.
    Odd cycles are forbidden by Lemma~\ref{lem:bipartitenecc}.
    So either $\ell, m, n$ must all be odd or $\ell, m, n$ must all be even.
    \begin{itemize}
        \item [\textbf{Case 1.}]
        $\ell, m, n$ are all even.
        Choose $\Delta_1 = \{a_1, a_2, v\}$
        and $\Delta_2 = \{a'_1, a'_2, u'\}$ in $S(\theta)$.
        Our $T K_{3, 3}$ is obtained on $\Delta_1 \sqcup \Delta_2$.
        We directly have the edges
        $a_1 a'_1$, $a_2 a'_2$, $a_2 a'_1$, $a_1 a'_2$ and $a_1 u'$.
        To complete the $T K_{3, 3}$,
        it remains to find mutually disjoint paths between vertices
        $u'v$, $u'a_2$, $a'_1v$ and $a'_2v$.
        We give the paths
        \begin{align*}
            P_{u'v}   &:= u', \xi_1, \xi'_2, \dots, \xi_{m-2}, v; \\
            P_{a'_1v} &:= a'_1, u, \xi'_1, \xi_2, \dots, \xi'_{m-2}, v; \\
            P_{a'_2v} &:= a'_2, a_3, \dots, a'_\ell, v; \\
            P_{u'a_2} &:= u', b_1, b_2, \dots, b_n, v', a_\ell, a'_{\ell-1}, \dots, a'_3, a_2.
        \end{align*}
        Observe that $P_{u'v}$ and $P_{a'_1v}$ alternate between vertices on
        $P_m = u, \xi_1, \dots, \xi_{m-2}, v$ and their shadow vertices.
        We also see that $P_{u'a_2}$ goes through all vertices of $P_n$,
        then meets $v'$,
        then alternates with $P_{a'_2v}$ on $P_\ell$ over $a_3, \dots, a_\ell$.
        Our choice of $\Delta_1$, $\Delta_2$, and these paths is
        such that no two pair of paths intersect.

    \item [\textbf{Case 2.}]
        $\ell, m, n$ are all odd.
        Choose $\Delta_1 = \{a_1, a_2, u\}$
        and $\Delta_2= \{a'_1, a'_2, v\}$ in $S(\theta)$.
        As before, our $T K_{3, 3}$ is obtained on $\Delta_1 \sqcup \Delta_2$.
        We directly have the edges
        $a_1 a'_1$, $a_2 a'_2$, $a_2 a'_1$, $a_1 a'_2$ and $a'_1 u$.
        It remains to find mutually disjoint paths between vertices
        $a_1 v$, $a_2v$, $u v$ and $a'_2u$.
        We give the paths
        \begin{align*}
            P_{u v} &:= u, \xi'_1, \xi_2, \dots, \xi_{m-2}, v; \\
            P_{a_1 v} &:= a_1, u', \xi_1, \xi'_2, \dots, \xi'_{m-2}, v; \\
            P_{a_2 v} &:= a_2, a'_3, \dots, a_\ell, v; \\
            P_{u a'_2} &:= u, b_1, b_2, \dots, b_n, v', a_\ell, a'_{\ell-1}, \dots, a_3, a'_2.
        \end{align*}
        The underlying structure is the same as in Case 1,
        but with a slightly different choice of vertices.
        The paths $P_{u v}$ and $P_{a_1 v}$ alternate between vertices on
        $P_m = u, \xi_1, \dots, \xi_{m-2}, v$ and their shadow vertices.
        Similar to the first case, the path $P_{u a'_2}$ starts at $u$,
        passes through all vertices of $P_n$,
        meets $v'$ as before,
        then alternates with $P_{a_2 v}$ to reach $a_2$.
        As before, these paths are disjoint.
    \end{itemize}
    In both cases, a $T K_{3, 3}$ is obtained, proving the claim.
\end{proof}

\begin{figure}[htbp]
    \centering
    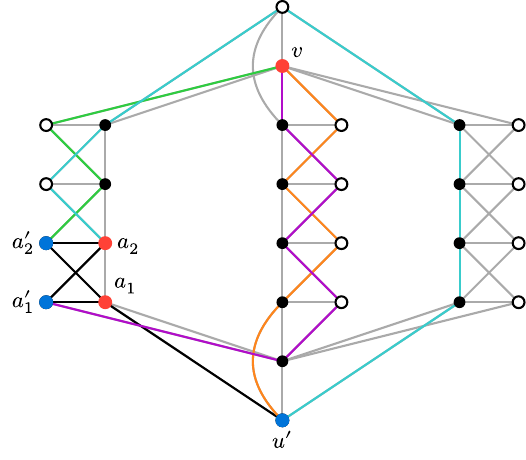
    \caption{The even case (Case 1) of Lemma~\ref{lem:cactusnecc}.
        The figure pictures $S(\theta (4, 6, 4))$.
        Non-gray edges highlight the embedded $T K_{3,3}$.
        White vertices are shadow vertices, and black vertices are originals.
        Blue vertices ($a'_2$, $a'_1$ and $u'$)
        and red vertices ($a_1$, $a_2$, $v$)
        form partitions $\Delta_2$ and $\Delta_1$ respectively.
        Black edges are the isolated edges $u'v$, $u'a_2$, $a'_1v$ and $a'_2v$.
        Colored edges lie on paths.
        The green path is $P_{a'_2v}$,
        the cyan path is $P_{u'a_2}$,
        the orange path is $P_{u'v}$,
        and the violet path is $P_{a'_1v}$.}
    \label{fig:evencase}
\end{figure}

\begin{figure}[htbp]
    \centering
    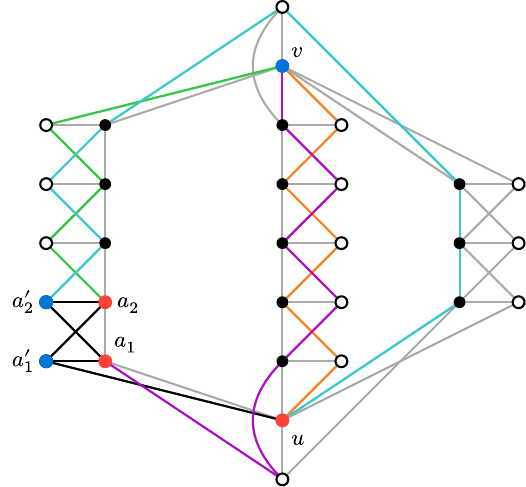
    \caption{The odd case (Case 2) of Lemma~\ref{lem:cactusnecc}.
        The figure pictures $S(\theta (5, 7, 3))$.
        As in Figure~\ref{fig:evencase}, the non-gray edges highlight the embedded $T K_{3,3}$.
        White vertices are shadow vertices, and black vertices are originals.
        Blue vertices (now $a'_1$, $a'_2$, and $v$)
        and red vertices (now $a_1$, $a_2$, $u$)
        form partitions $\Delta_2$ and $\Delta_1$ respectively.
        Black edges are the isolated edges $a_1 a'_1$, $a_2 a'_2$, $a_2 a'_1$, $a_1 a'_2$ and $a'_1 u$.
        Colored edges lie on paths.
        The green path is $P_{a_2v}$,
        the cyan path is $P_{a'_2u}$,
        the orange path is $P_{u v}$,
        and the violet path is $P_{a_1 v}$.}
    \label{fig:oddcase}
\end{figure}

Finally, we prove the necessary condition of Theorem~\ref{thm:main}.
\begin{lemma}
    Let $G$ be a graph. If $S(G)$ is planar, then $G$ must be a bipartite cactus.
\end{lemma}
\begin{proof}
    Let $G$ be a graph that is not a bipartite cactus.
    Then either $G$ contains an odd cycle, or $G$ contains a $T K_4^-$.
    If $G$ contains an odd cycle, then $S(G)$ contains $T K_{3, 3}$ by Lemma~\ref{lem:bipartitenecc}.
    Similarly, if $G$ contains a $T K_4^-$, then $S(G)$ contains a $T K_{3, 3}$ by Lemma~\ref{lem:cactusnecc}.
    In both cases, by Theorem~\ref{thm:kuratowski}, $S(G)$ cannot be planar.
\end{proof}

\section{Bipartite cactus is sufficient}\label{sec:bipartitecactussuff}

To prove sufficiency,
we provide a recursive algorithm to draw the great shadow of any bipartite cactus.
The even cycle is a base case of this algorithm,
and so we must first show that an even cycle has a planar great shadow.
\begin{lemma}\label{lem:evenplanar}
    If $C$ is an even cycle, then $S(C)$ is planar.
\end{lemma}
\begin{proof}
    We provide a planar drawing of $S(C)$ for any even cycle $C$.

    Fix any $v = v_1$, and number each vertex of $C$ from $1$ to $2n$
    in clockwise or counterclockwise order.
    Place each vertex of $C$ in this order on the boundary of a circle $O$
    with center $c$ and unit radius.
    For all $1 \leq k \leq 2n$,
    if $k$ is odd, place $v'_k$ on the exterior of $O$,
    on the line extended beyond $c$ and $v_k$, at distance $3/2$ from $c$.
    If $k$ is even, place $v'_k$ on the interior of $O$,
    on the line segment between $c$ and $v_k$, at distance $1/2$ from $c$.
    Finally, draw straight line segments
    between any two vertices connected by an edge in $S(C)$.
    The resulting drawing is planar by construction.
\end{proof}

\begin{figure}[htbp]
    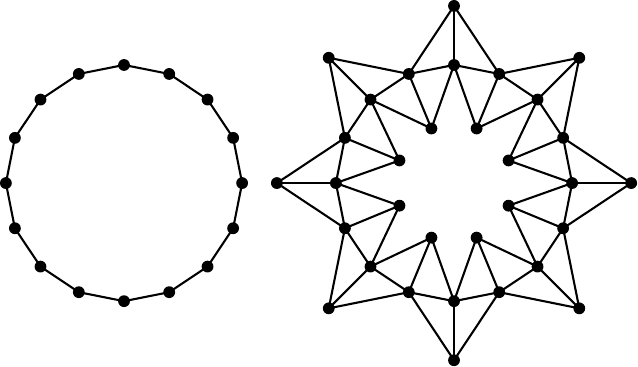
    \caption{The cycle graph $C_{16}$ is on the left.
        On the right is a a planar drawing of $S(C_{16})$,
        produced by the method described within the proof of Lemma~\ref{lem:evenplanar}.}
\end{figure}

\begin{remark}
    The constants $1/2$ and $3/2$ in the proof of Lemma~\ref{lem:evenplanar} are arbitrary.
    The argument still holds if
    each even shadow vertex is placed at any distance $d \in (0, 1)$,
    and
    each odd shadow vertex is placed at any distance $d \in (1, \infty)$.
\end{remark}

\begin{remark}
    The placement of the vertices of $C$ on the circle $O$ need not be
    uniform.
\end{remark}

Finally, we are ready to prove the sufficiency condition of Theorem~\ref{thm:main}.
This is done by providing a planar drawing of any bipartite cactus graph.

\begin{lemma}\label{lem:biglem}
    For any bipartite cactus graph $G$, there is a planar drawing of $S(G)$.
\end{lemma}
\begin{proof}
    Let $T$ be the cycle tree of $G$.
    We prove this statement by induction on $|T|$.
    In the base case, $T$ has one vertex.
    Then either $G$ has one vertex and $S(G)$ is an isolated edge,
    or $G$ is a cycle.
    If $G$ is a cycle, then by Lemma~\ref{lem:bipartitenecc},
    $G$ is a cycle of even order and a planar drawing
    of $S(G)$ is given by applying Lemma~\ref{lem:evenplanar}.

    Now suppose that Lemma~\ref{lem:biglem} holds for all cactus graphs with $|T| < k$.
    Further suppose $|T| > 1$.
    Choose some leaf $\ell$ of $T$. Let $p$ be the parent of $\ell$.
    Then,
    disconnect the edge $p \ell$.
    We obtain two smaller graphs:
    $G_\ell$, whose cycle tree is the isolated vertex $\ell$,
    and $G_p$,
    whose cycle tree is $T - \ell$.
    Then,
    by induction,
    we obtain a planar drawing of $S(G_\ell)$ and a planar drawing of $S(G_p)$.
    We now want to join these planar drawings together to obtain a planar drawing of $S(G)$.

    There are two cases to consider.
    If $p$ is a cycle vertex and $\ell$ is a cut vertex participating in the cycle of $p$,
    then $G_\ell$ and $G_p$ share a vertex $v$.
    So we must join the drawings of $S(G_\ell)$ and $S(G_p)$ on the edge $v v'$.
    Otherwise,
    if either $p$ is not a cycle vertex or $p$ is a cycle vertex and $\ell$ is a cut vertex not participating in $p$,
    then $G_\ell$ and $G_p$ are vertex-disjoint,
    and the disconnection of the edge $p \ell$ in $T$ corresponds to the disconnection of some edge $u v$ in $G$.
    Thus to join $S(G_\ell)$ and $S(G_p)$,
    we must introduce the edges $u v$, $u v'$ and $v u'$.

    \begin{itemize}
        \item [\textbf{Case 1.}] $G_\ell$ and $G_p$ share a vertex $v$.
        By Corollary~\ref{cor:exposededge},
        we can draw $S(G_\ell)$ and $S(G_p)$ to both expose the edge $v v'$.
        Let $R_\ell$ and $R_p$ be finite regions of the plane bounding the planar drawing of $S(G_\ell)$ and $S(G_p)$.
        As $v v'$ is on the exterior face of $S(G_\ell)$, one can place $v v'$ on the boundary of $R_\ell$.
        Similarly, one can place $v v'$ on the boundary of $R_p$.
        Then $R_r$ and $R_p$ can be placed on the plane such that both copies of $v v'$ overlap.
        This provides our desired drawing,
        seen in Figure~\ref{fig:edgeglue}.

    \item [\textbf{Case 2.}] $G_\ell$ and $G_p$ are vertex-disjoint.
        Let $u v$ be the edge deleted by disconnecting $p \ell$ (with $u \in G_\ell$ and $v \in G_p$).
        Use Corollary~\ref{cor:exposededge} to draw $S(G_\ell)$ (resp. $S(G_p)$) on the plane exposing the edge $u u'$ (resp. $v v'$).
        Let $R_\ell$ be a finite region of the plane
        bounding the planar drawing of $S(G_\ell)$.
        As $u$ and $u'$ are on the external face,
        one can place $u$ and $u'$ on the boundary of $R_\ell$.
        Similarly let $R_p$ be the region bounding the planar drawing of $S(G_p)$,
        with $v$ and $v'$ on its boundary.

        We construct $S(G)$ from $S(G_\ell)$ and $S(G_p)$ by
        introducing the edges $u v$, $u v'$ and $v u'$.
        The three edges can be drawn with $u v'$ and $v u'$ enclosing $u v$.
        The resulting drawing is planar,
        which is seen in Figure~\ref{fig:treeind}.
    \end{itemize}
    In both cases,
    a planar drawing of $S(G)$ is obtained,
    as desired.
\end{proof}

\begin{figure}[htbp]
    \centering
    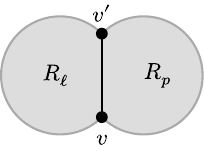
    \caption{Case 1 of Lemma~\ref{lem:biglem}.
        We adjoin the drawings of $S(G_\ell)$ and $S(G_p)$,
        each wholly contained within regions $R_\ell$ and $R_p$,
        by overlapping the common edge $v v'$.}
    \label{fig:edgeglue}
\end{figure}

\begin{figure}[htbp]
    \centering
    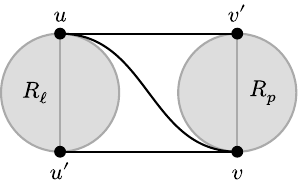
    \caption{Case 2 of Lemma~\ref{lem:biglem}.
        We adjoin the drawings of $S(G_\ell)$ and $S(G_p)$,
        each wholly contained within regions $R_\ell$ and $R_p$,
        by introducing the edges $u v'$, $v u'$ and $u v$.}
    \label{fig:treeind}
\end{figure}

\section*{Acknowledgements}
I am thankful to
Christiaan van de Sande
for an early discussion on the necessary conditions of Theorem~\ref{thm:main},
to
Xingxing Yu
for Remark~\ref{rmk:altbipartitenecc},
and to
Pemma Reiter,
Kirtus Leyba,
and Jamie Waitt
for feedback that improved the exposition in the introduction and preliminaries.
I am especially thankful to
Joshua J. Daymude
and to
Zilin Jiang
for their extensive feedback and corrections to this document.

\bibliographystyle{plain}
\bibliography{refs}

\end{document}

%% file: figure01_svg-tex.pdf_tex
\begingroup%
  \makeatletter%
  \providecommand\color[2][]{%
    \errmessage{(Inkscape) Color is used for the text in Inkscape, but the package 'color.sty' is not loaded}%
    \renewcommand\color[2][]{}%
  }%
  \providecommand\transparent[1]{%
    \errmessage{(Inkscape) Transparency is used (non-zero) for the text in Inkscape, but the package 'transparent.sty' is not loaded}%
    \renewcommand\transparent[1]{}%
  }%
  \providecommand\rotatebox[2]{#2}%
  \newcommand*\fsize{\dimexpr\f@size pt\relax}%
  \newcommand*\lineheight[1]{\fontsize{\fsize}{#1\fsize}\selectfont}%
  \ifx\svgwidth\undefined%
    \setlength{\unitlength}{171.89764404bp}%
    \ifx\svgscale\undefined%
      \relax%
    \else%
      \setlength{\unitlength}{\unitlength * \real{\svgscale}}%
    \fi%
  \else%
    \setlength{\unitlength}{\svgwidth}%
  \fi%
  \global\let\svgwidth\undefined%
  \global\let\svgscale\undefined%
  \makeatother%
  \begin{picture}(1,0.6042218)%
    \lineheight{1}%
    \setlength\tabcolsep{0pt}%
    \put(0,0){\includegraphics[width=\unitlength,page=1]{figure01_svg-tex.pdf}}%
  \end{picture}%
\endgroup%

%% file: figure02_svg-tex.pdf_tex
\begingroup%
  \makeatletter%
  \providecommand\color[2][]{%
    \errmessage{(Inkscape) Color is used for the text in Inkscape, but the package 'color.sty' is not loaded}%
    \renewcommand\color[2][]{}%
  }%
  \providecommand\transparent[1]{%
    \errmessage{(Inkscape) Transparency is used (non-zero) for the text in Inkscape, but the package 'transparent.sty' is not loaded}%
    \renewcommand\transparent[1]{}%
  }%
  \providecommand\rotatebox[2]{#2}%
  \newcommand*\fsize{\dimexpr\f@size pt\relax}%
  \newcommand*\lineheight[1]{\fontsize{\fsize}{#1\fsize}\selectfont}%
  \ifx\svgwidth\undefined%
    \setlength{\unitlength}{271.11022949bp}%
    \ifx\svgscale\undefined%
      \relax%
    \else%
      \setlength{\unitlength}{\unitlength * \real{\svgscale}}%
    \fi%
  \else%
    \setlength{\unitlength}{\svgwidth}%
  \fi%
  \global\let\svgwidth\undefined%
  \global\let\svgscale\undefined%
  \makeatother%
  \begin{picture}(1,0.23002528)%
    \lineheight{1}%
    \setlength\tabcolsep{0pt}%
    \put(0,0){\includegraphics[width=\unitlength,page=1]{figure02_svg-tex.pdf}}%
  \end{picture}%
\endgroup%

%% file: figure03_svg-tex.pdf_tex
\begingroup%
  \makeatletter%
  \providecommand\color[2][]{%
    \errmessage{(Inkscape) Color is used for the text in Inkscape, but the package 'color.sty' is not loaded}%
    \renewcommand\color[2][]{}%
  }%
  \providecommand\transparent[1]{%
    \errmessage{(Inkscape) Transparency is used (non-zero) for the text in Inkscape, but the package 'transparent.sty' is not loaded}%
    \renewcommand\transparent[1]{}%
  }%
  \providecommand\rotatebox[2]{#2}%
  \newcommand*\fsize{\dimexpr\f@size pt\relax}%
  \newcommand*\lineheight[1]{\fontsize{\fsize}{#1\fsize}\selectfont}%
  \ifx\svgwidth\undefined%
    \setlength{\unitlength}{252.26118469bp}%
    \ifx\svgscale\undefined%
      \relax%
    \else%
      \setlength{\unitlength}{\unitlength * \real{\svgscale}}%
    \fi%
  \else%
    \setlength{\unitlength}{\svgwidth}%
  \fi%
  \global\let\svgwidth\undefined%
  \global\let\svgscale\undefined%
  \makeatother%
  \begin{picture}(1,0.8750154)%
    \lineheight{1}%
    \setlength\tabcolsep{0pt}%
    \put(0,0){\includegraphics[width=\unitlength,page=1]{figure03_svg-tex.pdf}}%
  \end{picture}%
\endgroup%

%% file: figure04_svg-tex.pdf_tex
\begingroup%
  \makeatletter%
  \providecommand\color[2][]{%
    \errmessage{(Inkscape) Color is used for the text in Inkscape, but the package 'color.sty' is not loaded}%
    \renewcommand\color[2][]{}%
  }%
  \providecommand\transparent[1]{%
    \errmessage{(Inkscape) Transparency is used (non-zero) for the text in Inkscape, but the package 'transparent.sty' is not loaded}%
    \renewcommand\transparent[1]{}%
  }%
  \providecommand\rotatebox[2]{#2}%
  \newcommand*\fsize{\dimexpr\f@size pt\relax}%
  \newcommand*\lineheight[1]{\fontsize{\fsize}{#1\fsize}\selectfont}%
  \ifx\svgwidth\undefined%
    \setlength{\unitlength}{252.26118469bp}%
    \ifx\svgscale\undefined%
      \relax%
    \else%
      \setlength{\unitlength}{\unitlength * \real{\svgscale}}%
    \fi%
  \else%
    \setlength{\unitlength}{\svgwidth}%
  \fi%
  \global\let\svgwidth\undefined%
  \global\let\svgscale\undefined%
  \makeatother%
  \begin{picture}(1,0.92539385)%
    \lineheight{1}%
    \setlength\tabcolsep{0pt}%
    \put(0,0){\includegraphics[width=\unitlength,page=1]{figure04_svg-tex.pdf}}%
  \end{picture}%
\endgroup%

%% file: figure05_svg-tex.pdf_tex
\begingroup%
  \makeatletter%
  \providecommand\color[2][]{%
    \errmessage{(Inkscape) Color is used for the text in Inkscape, but the package 'color.sty' is not loaded}%
    \renewcommand\color[2][]{}%
  }%
  \providecommand\transparent[1]{%
    \errmessage{(Inkscape) Transparency is used (non-zero) for the text in Inkscape, but the package 'transparent.sty' is not loaded}%
    \renewcommand\transparent[1]{}%
  }%
  \providecommand\rotatebox[2]{#2}%
  \newcommand*\fsize{\dimexpr\f@size pt\relax}%
  \newcommand*\lineheight[1]{\fontsize{\fsize}{#1\fsize}\selectfont}%
  \ifx\svgwidth\undefined%
    \setlength{\unitlength}{305.80316162bp}%
    \ifx\svgscale\undefined%
      \relax%
    \else%
      \setlength{\unitlength}{\unitlength * \real{\svgscale}}%
    \fi%
  \else%
    \setlength{\unitlength}{\svgwidth}%
  \fi%
  \global\let\svgwidth\undefined%
  \global\let\svgscale\undefined%
  \makeatother%
  \begin{picture}(1,0.57470966)%
    \lineheight{1}%
    \setlength\tabcolsep{0pt}%
    \put(0,0){\includegraphics[width=\unitlength,page=1]{figure05_svg-tex.pdf}}%
  \end{picture}%
\endgroup%

%% file: figure06_svg-tex.pdf_tex
\begingroup%
  \makeatletter%
  \providecommand\color[2][]{%
    \errmessage{(Inkscape) Color is used for the text in Inkscape, but the package 'color.sty' is not loaded}%
    \renewcommand\color[2][]{}%
  }%
  \providecommand\transparent[1]{%
    \errmessage{(Inkscape) Transparency is used (non-zero) for the text in Inkscape, but the package 'transparent.sty' is not loaded}%
    \renewcommand\transparent[1]{}%
  }%
  \providecommand\rotatebox[2]{#2}%
  \newcommand*\fsize{\dimexpr\f@size pt\relax}%
  \newcommand*\lineheight[1]{\fontsize{\fsize}{#1\fsize}\selectfont}%
  \ifx\svgwidth\undefined%
    \setlength{\unitlength}{97.77457428bp}%
    \ifx\svgscale\undefined%
      \relax%
    \else%
      \setlength{\unitlength}{\unitlength * \real{\svgscale}}%
    \fi%
  \else%
    \setlength{\unitlength}{\svgwidth}%
  \fi%
  \global\let\svgwidth\undefined%
  \global\let\svgscale\undefined%
  \makeatother%
  \begin{picture}(1,0.74010872)%
    \lineheight{1}%
    \setlength\tabcolsep{0pt}%
    \put(0,0){\includegraphics[width=\unitlength,page=1]{figure06_svg-tex.pdf}}%
  \end{picture}%
\endgroup%

%% file: figure07_svg-tex.pdf_tex
\begingroup%
  \makeatletter%
  \providecommand\color[2][]{%
    \errmessage{(Inkscape) Color is used for the text in Inkscape, but the package 'color.sty' is not loaded}%
    \renewcommand\color[2][]{}%
  }%
  \providecommand\transparent[1]{%
    \errmessage{(Inkscape) Transparency is used (non-zero) for the text in Inkscape, but the package 'transparent.sty' is not loaded}%
    \renewcommand\transparent[1]{}%
  }%
  \providecommand\rotatebox[2]{#2}%
  \newcommand*\fsize{\dimexpr\f@size pt\relax}%
  \newcommand*\lineheight[1]{\fontsize{\fsize}{#1\fsize}\selectfont}%
  \ifx\svgwidth\undefined%
    \setlength{\unitlength}{142.73228455bp}%
    \ifx\svgscale\undefined%
      \relax%
    \else%
      \setlength{\unitlength}{\unitlength * \real{\svgscale}}%
    \fi%
  \else%
    \setlength{\unitlength}{\svgwidth}%
  \fi%
  \global\let\svgwidth\undefined%
  \global\let\svgscale\undefined%
  \makeatother%
  \begin{picture}(1,0.62332631)%
    \lineheight{1}%
    \setlength\tabcolsep{0pt}%
    \put(0,0){\includegraphics[width=\unitlength,page=1]{figure07_svg-tex.pdf}}%
  \end{picture}%
\endgroup%